\documentclass[a4paper,english,fontsize=11pt,parskip=half,abstracton]{scrartcl}
\usepackage{babel}
\usepackage[utf8]{inputenc}
\usepackage[T1]{fontenc}
\usepackage[a4paper,left=22mm,right=22mm,top=30mm,bottom=30mm]{geometry}
\usepackage{amsmath}
\usepackage{amsthm}
\usepackage{amssymb}
\usepackage{mathtools}
\usepackage{enumerate}
\usepackage{aliascnt}
\usepackage[bookmarks=false,
            pdftitle={Bounding the number of characters in a block of a finite group},
            pdfauthor={Benjamin Sambale},
            pdfkeywords={number of characters, blocks, Cartan matrix, Brauer},
            pdfstartview={FitH}]{hyperref}

\newtheorem*{ThmA}{Theorem~A}
\newtheorem*{ThmB}{Theorem~B}
\newaliascnt{Lem}{Thm}
\newtheorem{Lem}[Lem]{Lemma}
\aliascntresetthe{Lem}
\newaliascnt{Prop}{Thm}
\newtheorem{Prop}[Prop]{Proposition}
\aliascntresetthe{Prop}
\newaliascnt{Cor}{Thm}
\newtheorem{Cor}[Cor]{Corollary}
\aliascntresetthe{Cor}
\newaliascnt{Con}{Thm}

\aliascntresetthe{Con}
\theoremstyle{definition}
\newaliascnt{Def}{Thm}
\newtheorem{Def}[Def]{Definition}
\aliascntresetthe{Def}
\newaliascnt{Ex}{Thm}

\aliascntresetthe{Ex}

\allowdisplaybreaks[1]

\renewcommand{\phi}{\varphi}
\newcommand{\C}{\mathrm{C}}
\newcommand{\N}{\mathrm{N}}
\newcommand{\Z}{\mathrm{Z}}
\newcommand{\ZZ}{\mathbb{Z}}
\newcommand{\CC}{\mathbb{C}}
\newcommand{\QQ}{\mathbb{Q}}
\newcommand{\RR}{\mathbb{R}}
\newcommand{\NN}{\mathbb{N}}

\newcommand{\CG}{\mathcal{G}}
\newcommand{\CN}{\mathcal{N}}
\newcommand{\CA}{\mathcal{A}}

\newcommand{\CW}{\mathcal{W}}
\newcommand{\CS}{\mathcal{S}}

\newcommand{\Aut}{\operatorname{Aut}}
\newcommand{\AGammaL}{\operatorname{A\Gamma L}}

\newcommand{\GL}{\operatorname{GL}}
\newcommand{\SL}{\operatorname{SL}}

\newcommand{\Irr}{\operatorname{Irr}}
\newcommand{\IBr}{\operatorname{IBr}}

\newcommand{\Gal}{\operatorname{Gal}}

\newcommand{\tr}{\operatorname{tr}}

\mathchardef\ordinarycolon\mathcode`\:  %defines a nice ":=" 
 \mathcode`\:=\string"8000
 \begingroup \catcode`\:=\active
   \gdef:{\mathrel{\mathop\ordinarycolon}}
 \endgroup

\title{Bounding the number of characters\\ in a block of a finite group}
\author{Benjamin Sambale\footnote{Fachbereich Mathematik, TU Kaiserslautern, 67653 Kaiserslautern, Germany, 
\href{mailto:sambale@mathematik.uni-kl.de}{sambale@mathematik.uni-kl.de}\newline
Institut für Mathematik, FSU Jena, 07737 Jena, Germany, \href{mailto:benjamin.sambale@uni-jena.de}{benjamin.sambale@uni-jena.de}}}
\date{\today}

\begin{document}
\frenchspacing
\maketitle
\begin{abstract}\noindent
We present a strong upper bound on the number $k(B)$ of irreducible characters of a $p$-block $B$ of a finite group $G$ in terms of local invariants. More precisely, the bound depends on a chosen major $B$-subsection $(u,b)$, its normalizer $\N_G(\langle u\rangle,b)$ in the fusion system and a weighted sum of the Cartan invariants of $b$.
In this way we strengthen and unify previous bounds given by Brauer, Wada, Külshammer--Wada, Héthelyi--Külshammer--Sambale and the present author. 
\end{abstract}

\textbf{Keywords:} number of characters in a block, Cartan matrix, Brauer's $k(B)$-Conjecture\\
\textbf{AMS classification:} 20C15, 20C20

\section{Introduction}

Let $B$ be a $p$-block of a finite group $G$ with defect $d$. Since Richard Brauer~\cite{Brauer46} conjectured that the number of irreducible characters $k(B)$ in $B$ is at most $p^d$, there has been great interest in bounding $k(B)$ in terms of local invariants. Brauer and Feit~\cite{BrauerFeit} used some properties of the Cartan matrix $C=(c_{ij})\in\ZZ^{l(B)\times l(B)}$ of $B$ to prove their celebrated bound $k(B)\le p^{2d}$ (here and in the following $l(B)$ denotes the number of irreducible Brauer characters of $B$). In the present paper we investigate stronger bounds by making use of further local invariants. By elementary facts on decomposition numbers, it is easy to see that 
\begin{equation}\label{trace}
k(B)\le\tr(C)
\end{equation}
where $\tr(C)$ denotes the trace of $C$. However, it is not true in general that $\tr(C)\le p^d$. In fact, there are examples with $\tr(C)>l(B)p^d$ (see \cite{NS}) although Brauer already knew that $k(B)\le l(B)p^d$ (see \autoref{brauer} below) and this was subsequently improved by Olsson~\cite[Theorem~4]{OlssonIneq}. 
For this reason, some authors strengthened \eqref{trace} in a number of ways. Brandt~\cite[Proposition~4.2]{Brandt} proved
\[k(B)\le\tr(C)-l(B)+1\]
and this was generalized by the present author in \cite[Proposition~8]{SambaleC3} to
\[k(B)\le\sum_{i=1}^m\det(C_i)-m+1\]
where $S_1,\ldots,S_m$ is a partition of $\{1,\ldots,l(B)\}$ and $C_i:=(c_{st})_{s,t\in S_i}$. 
Using different methods, Wada~\cite{Wada6} observed that
\begin{equation}\label{W}
k(B)\le \tr(C)-\sum_{i=1}^{l(B)-1}c_{i,i+1}.
\end{equation}
In Külshammer--Wada~\cite{KuelshammerWada}, the authors noted that \eqref{W} is a special case of
\begin{equation}\label{KW}
k(B)\le\sum_{1\le i\le j\le l(B)}q_{ij}c_{ij}
\end{equation}
where $q(x)=\sum_{1\le i\le j\le l(B)}q_{ij}x_ix_j$ is a (weakly) positive definite integral quadratic form.

Since $C$ is often harder to compute than $k(B)$, it is desirable to replace $C$ by the Cartan matrix of a Brauer correspondent of $B$ in a proper subgroup. For this purpose let $D$ be a defect group of $B$ and choose $u\in\Z(D)$. Then a Brauer correspondent $b$ of $B$ in $\C_G(u)$ has defect group $D$ as well. 
The present author replaced $c_{ij}$ in \eqref{KW} by the corresponding entries of the Cartan matrix $C_u$ of $b$ (see \cite[Lemma~1]{SambalekB2}).

In Héthelyi--Külshammer--Sambale~\cite[Theorem~2.4]{HKS} we have invoked Galois actions to obtain stronger bounds although only in the special cases $p=2$ and $l(b)=1$ (see \cite[Theorems~3.1 and 4.10]{HKS}). More precisely, in the latter case we proved
\begin{equation}\label{HKS}
k(B)\le\sum_{i=1}^\infty p^{2i}k_i(B)\le\Bigl(n+\frac{|\langle u\rangle|-1}{n}\Bigr)\frac{p^d}{|\langle u\rangle|}\le p^d=\tr(C_u)
\end{equation}
where $n:=\lvert\N_G(\langle u\rangle,b):\C_G(u)|$ and $k_i(B)$ is the number of irreducible characters of height $i\ge 0$ in $B$. This is a refinement of a result of Robinson~\cite[Theorem~3.4.3]{RobinsonNumber}.
In \cite[Theorem~2.6]{Sambalerefine}, the present author relaxed the condition $l(b)=1$ to the weaker requirement that $\CN:=\N_G(\langle u\rangle,b)/\C_G(u)$ acts trivially on the set $\IBr(b)$ of irreducible Brauer characters of $b$.

In this paper we replace quadratic forms by a matrix $W$ of weights which allows us to drop all restrictions imposed above.
We prove the following general result which incorporates the previous special cases (see Section~3 for details).

\begin{ThmA}\label{gold}
Let $B$ be a block of a finite group $G$ with defect group $D$. Let $u\in\Z(D)$ and let $b$ be a Brauer correspondent of $B$ in $\C_G(u)$. Let $\mathcal{N}:=\N_G(\langle u\rangle,b)/\C_G(u)$ and let $C$ be the Cartan matrix of the block $\overline{b}$ of $\C_G(u)/\langle u\rangle$ dominated by $b$. Let $W\in\RR^{l(b)\times l(b)}$ such that $xWx^\text{t}\ge 1$ for every $x\in\ZZ^{l(b)}\setminus\{0\}$. Then
\[\boxed{k(B)\le \Bigl(|\mathcal{N}|+\frac{|\langle u\rangle|-1}{|\mathcal{N}|}\Bigr)\tr(WC)\le|\langle u\rangle|\tr(WC).}\]
The first inequality is strict if $\mathcal{N}$ acts non-trivially on $\IBr(b)$ and the second inequality is strict if and only if $1<|\mathcal{N}|<|\langle u\rangle|-1$.
\end{ThmA}

In contrast to \eqref{HKS}, we cannot replace $k(B)$ by $\sum p^{2i}k_i(B)$ in Theorem~A (the principal $2$-block of $\SL(2,3)$ is a counterexample with $u=1$).
By a classical fusion argument of Burnside, the group $\CN$ in Theorem~A is induced from the inertial quotient $\N_G(D,b_D)/D\C_G(D)$ where $b_D$ is a Brauer correspondent of $B$ in $\C_G(D)$ (see \cite[Corollary~4.18]{LocalMethods}). In particular, $\CN$ is a $p'$-group and $|\CN|$ divides $p-1$.

As noted in previous papers, if $u\in D\setminus\Z(D)$, one still gets upper bounds on the number of height $0$ characters and this is of interest with respect to Olsson's Conjecture $k_0(B)\le|D:D'|$ where $D'$ denotes the commutator subgroup of $D$. In fact, we will deduce Theorem~A from our second main theorem:

\begin{ThmB}\label{h0}
Let $B$ be a block of a finite group $G$ with defect group $D$. Let $u\in D$ and let $b$ be a Brauer correspondent of $B$ in $\C_G(u)$. Let $\mathcal{N}:=\N_G(\langle u\rangle,b)/\C_G(u)$ and let $C$ be the Cartan matrix of the block $\overline{b}$ of $\C_G(u)/\langle u\rangle$ dominated by $b$. Let $W\in\RR^{l(b)\times l(b)}$ such that $xWx^\text{t}\ge 1$ for every $x\in\ZZ^{l(b)}\setminus\{0\}$. Then
\[\boxed{k_0(B)\le k_0\bigl(\langle u\rangle\rtimes\CN\bigr)\tr(WC)\le|\langle u\rangle|\tr(WC).}\]
The first inequality is strict if $\mathcal{N}$ acts non-trivially on $\IBr(b)$. 
\end{ThmB}

In the situation of Theorem~B we may assume, after conjugation, that $\N_D(\langle u\rangle)/\C_D(u)$ is a Sylow $p$-subgroup of $\CN$ (see \cite[Proposition~2.5]{AKO}). In particular, $\CN=\N_D(\langle u\rangle)/\C_D(u)$ whenever $p=2$.

If $\CN$ acts trivially on $\IBr(b)$, then our bounds cannot be improved in general. To see this, let $\langle u\rangle$ be any cyclic $p$-group, and let $\CN\le\Aut(\langle u\rangle)$. Then $G:=\langle u\rangle\rtimes\CN$ has only one $p$-block $B$. In this situation $l(b)=1$ and $C=(1)$. Hence, $k_0(B)=k_0(G)=k_0(\langle u\rangle\rtimes\CN)\tr(WC)$ for $W=(1)$. Similarly, if $\CN$ is a $p'$-group, then $k(B)=k(G)=|\CN|+\frac{|\langle u\rangle|-1}{|\CN|}$. 

It is known that the ordinary character table of $\C_G(u)/\langle u\rangle$ determines $C$ up to basic sets, i.\,e. up to transformations of the form $S^\text{t}CS$ where $S\in\GL(l(b),\ZZ)$ and $S^\text{t}$ is the transpose of $S$. Then $\widetilde{W}:=S^{-1}WS^{-\text{t}}$ still satisfies  $x\widetilde{W}x^\text{t}\ge 1$ for every $x\in\ZZ^{l(b)}\setminus\{0\}$ and 
\[\tr(\widetilde{W}S^\text{t}CS)=\tr(S^{-1}WCS)=\tr(WC).\] 
Hence, our results do not depend on the chosen basic set. 

\section{Proofs}

First we outline the proof of Theorem~B: For sake of simplicity suppose first that $u=1$. Then every row $d_\chi$ of the decomposition matrix $Q$ of $B$ is non-zero and $Q^\text{t}Q=C$. Hence, 
\[k(B)\le\sum_{\chi\in\Irr(B)}d_\chi Wd_\chi^\text{t}=\tr(QWQ^\text{t})=\tr(WQ^\text{t}Q)=\tr(WC).\]
In the general case we replace $Q$ be the generalized decomposition matrix with respect to the subsection $(u,b)$. Then $Q$ consists of algebraic integers in the cyclotomic field of degree $q:=|\langle u\rangle|$. We apply a discrete Fourier transformation to turn $Q$ into an integral matrix with the same number of rows, but with more columns. At the same time we need to blow up $W$ to a larger matrix with similar properties. Afterwards we use the fact that the rows of $Q$ corresponding to height $0$ characters are non-zero and fulfill a certain $p$-adic valuation. 
For $p=2$ the proof can be completed directly, while for $p>2$ we argue by induction on $q$. Additional arguments are required to handle the case where $|\CN|$ is divisible by $p$. These calculations make use of sophisticated matrix analysis.

We fix the following matrix notation. For $n\in\NN$ let $1_n$ be the identity matrix of size $n\times n$ and similarly let $0_n$ be the zero matrix of the same size. 
Moreover, let
\[U_n:=\frac{1}{2}\begin{pmatrix}
2&-1&&0\\
-1&\ddots&\ddots\\
&\ddots&\ddots&-1\\
0&&-1&2
\end{pmatrix}\in\QQ^{n\times n}.\]
For $d\in\NN$ let $d^{n\times n}$ be the $n\times n$ matrix which has every entry equal to $d$. For $A\in\RR^{n\times n}$ and $B\in\RR^{m\times m}$ we construct the direct sum $A\oplus B\in\RR^{(n+m)\times(n+m)}$ and the Kronecker product $A\otimes B\in\RR^{nm\times nm}$ in the usual manner. Note that $\tr(A\oplus B)=\tr(A)+\tr(B)$ and $\tr(A\otimes B)=\tr(A)\tr(B)$.  
Finally, let $\delta_{ij}$ be the Kronecker delta. We assume that every positive (semi)definite matrix is symmetric. Moreover, we call a symmetric matrix $A\in\RR^{n\times n}$ \emph{integral positive definite}, if $xAx^\text{t}\ge 1$ for every $x\in\RR^n\setminus\{0\}$.

The proof of Theorem~B is deduced from a series of lemmas and propositions. 

\begin{Lem}\label{lemsym}
Every integral positive definite matrix is positive definite.
\end{Lem}
\begin{proof}
Let $W\in\RR^{n\times n}$ be integral positive definite.
By way of contradiction, suppose that there exists an eigenvector $v\in\RR^n$ of $W$ with eigenvalue $\lambda\le 0$ and (euclidean) norm $1$. 
If $\lambda<0$, choose $x\in\QQ^n$ such that $\lVert x\rVert\le\lVert v\rVert=1$ and $\lVert x-v\rVert<-\frac{\lambda}{2\lVert W\rVert}$ where $\lVert W\rVert$ denotes the Frobenius matrix norm of $W$. Then
\[xWx^\text{t}=(x-v)W(x+v)^\text{t}+vWv^\text{t}\le\lVert x-v\rVert\lVert W\rVert\lVert x+v\rVert+\lambda<0.\]
However, there exists $m\in\NN$ such that $mx\in\ZZ^n$ and $(mx)W(mx)^\text{t}<0$. This contradiction implies $\lambda=0$.
By Dirichlet's approximation theorem (see \cite[Theorem~200]{Hardy}) there exist infinitely many integers $m$ and $x\in\ZZ^n$ such that 
\[\lVert x-mv\rVert<\frac{\sqrt{n}}{\sqrt[n]{m}}.\]
It follows that 
\[xWx^\text{t}=(x-mv)W(x-mv)^\text{t}\le\lVert x-mv\rVert^2\lVert W\rVert<1\]
if $m$ is sufficiently large. Again we have a contradiction.
\end{proof}

Conversely, every positive definite matrix can be scaled to an integral positive definite matrix.
The next lemma is a key argument when dealing with non-trivial actions of $\CN$ on $\IBr(b)$.

\begin{Lem}\label{lemposperm}
Let $A,B\in\RR^{n\times n}$ positive semidefinite matrices such that $A$ commutes with a permutation matrix $P\in\RR^{n\times n}$. Then $\tr(ABP)\le\tr(AB)$. If $A$ and $B$ are positive definite, then $\tr(ABP)=\tr(AB)$ if and only if $P=1_n$.
\end{Lem}
\begin{proof}
By the spectral theorem, $A$ and $P$ are diagonalizable. Since they commute, they are simultaneously diagonalizable. Since $A$ has real, non-negative eigenvalues, there exists a positive semidefinite matrix $A^{1/2}\in\RR^{n\times n}$ such that $A^{1/2}A^{1/2}=A$ and $A^{1/2}P=PA^{1/2}$. Then $M:=(m_{ij})=A^{1/2}BA^{1/2}$ is also positive semidefinite. In particular $m_{ij}\le(m_{ii}+m_{jj})/2$ for $i,j\in\{1,\ldots,n\}$. If $\sigma$ denotes the permutation corresponding to $P$, then we obtain
\[\tr(ABP)=\tr(A^{1/2}BPA^{1/2})=\tr(MP)=\sum_{i=1}^nm_{i\sigma(i)}\le\sum_{i=1}^n\frac{m_{ii}+m_{\sigma(i)\sigma(i)}}{2}=\tr(M)=\tr(AB).\]
If $A$ and $B$ are positive definite, then so is $M$ and we have $m_{ij}<(m_{ii}+m_{jj})/2$ whenever $i\ne j$. This implies the last claim.
\end{proof}

\begin{Lem}\label{lemwn}
Let $W\in\RR^{n\times n}$ be integral positive definite and suppose that $W$ commutes with a permutation matrix $P$. Let
\[
W_m:=\frac{1}{2}\begin{pmatrix}
2W&-P W&&0\\
-P^\text{t}W&\ddots&\ddots\\
&\ddots&\ddots&-P W\\
0&&-P^\text{t}W&2W
\end{pmatrix}\in\RR^{mn\times mn}.
\]
Then $W_m$ is integral positive definite. In particular, $U_m\otimes W$ is integral positive definite.
\end{Lem}
\begin{proof}
Let $x=(x_1,\ldots,x_m)$ with $x_i\in\ZZ^n$. Since $WP=PW$ we have
\begin{align*}
xW_mx^\text{t}&=\sum_{i=1}^mx_iWx_i^\text{t}-\sum_{i=1}^{m-1}x_iP Wx_{i+1}^\text{t}\\
&=\frac{1}{2}x_1Wx_1^\text{t}+\frac{1}{2}x_mWx_m^\text{t}+\frac{1}{2}\sum_{i=1}^{m-1}(x_iP-x_{i+1})W(x_iP-x_{i+1})^\text{t}.
\end{align*}
We may assume that $x_i\ne 0$ for some $i\in\{1,\ldots,m\}$. If $i=1$, then $x_m\ne 0$ or $x_jP\ne x_{j+1}$ for some $j$. In any case $xW_mx^\text{t}\ge 1$. If $i>1$, then the claim can be seen in a similar fashion. The last claim follows with $P=1_n$.
\end{proof}

Now assume the notation of Theorem~B. 
In addition, let $p$ be the characteristic of $B$ such that $q:=|\langle u\rangle|$ is a power of $p$. 
Let $k:=k(B)$, $l:=l(b)$ and $\zeta:=e^{2\pi i/q}\in\CC$. Then the generalized decomposition matrix $Q=(d_{\chi\phi}^u)$ of $B$ with respect to the subsection $(u,b)$ has size $k\times l$ and entries in $\ZZ[\zeta]$ (see \cite[Definition~1.19]{habil} for instance). 
By the orthogonality relations of generalized decomposition numbers, we have $Q^\text{t}\overline{Q}=qC$ where $qC$ is the Cartan matrix of $b$ (see \cite[Theorems~1.14 and 1.22]{habil}). Recall that $C$ is positive definite and has non-negative integer entries.

The first part of the next lemma is a result of Broué~\cite{BroueSanta} while the second part was known to Brauer~\cite[(5H)]{BrauerBlSec2}.

\begin{Lem}[{\cite[Proposition~1.36]{habil}}]\label{padic}
Let $d_\chi$ be a row of $Q$ corresponding to a character $\chi\in\Irr(B)$ of height $0$. Let $d$ be the defect of $\overline{b}$ and let $\widetilde{C}:=p^dC^{-1}\in\ZZ^{l\times l}$. Then the $p$-adic valuation of $d_\chi\widetilde{C}\overline{d_\chi}^\textnormal{t}$ is $0$. In particular, $d_\chi\ne 0$. 
Now assume that $u\in\Z(D)$ and $\chi\in\Irr(B)$ is arbitrary. Then $d_\chi\ne 0$.
\end{Lem}

We identify the Galois group $\mathcal{G}:=\Gal(\QQ(\zeta)|\QQ)$ with $\Aut(\langle u\rangle)\cong(\ZZ/q\ZZ)^\times$ such that $\gamma(\zeta)=\zeta^\gamma$ for $\gamma\in\mathcal{G}$. In this way we regard $\mathcal{N}$ as a subgroup of $\mathcal{G}$.
Let $n:=|\CN|$.
For any $\gamma\in\mathcal{G}$, $\gamma(Q)$ is the generalized decomposition matrix with respect to $(u^\gamma,b)$. If the subsections $(u,b)$ and $(u^\gamma,b)$ are not conjugate in $G$, then $\gamma\notin\mathcal{N}$ and $\gamma(Q)^\text{t}\overline{Q}=0$. On the other hand, if they are conjugate, then $\gamma\in\mathcal{N}$ and 
\begin{equation}\label{relation}
\gamma(d_{\chi\phi}^u)=d^{u^\gamma}_{\chi\phi}=d^u_{\chi\phi^\gamma}
\end{equation}
for $\chi\in\Irr(B)$ and $\phi\in\IBr(b)$. Hence, in this case, $\gamma$ acts on the columns of $Q$ and there exists a permutation matrix $P_\gamma$ such that $\gamma(Q)=QP_\gamma$. Recall that permutation matrices are orthogonal, i.\,e. $P_{\gamma^{-1}}=P_\gamma^{-1}=P_\gamma^\text{t}$. 
Since $\mathcal{G}$ is abelian, we obtain
\begin{equation}\label{CPPC}
CP_\gamma=Q^\text{t}\gamma(\overline{Q})=\gamma^{-1}(Q)^\text{t}\overline{Q}=P_{\gamma^{-1}}^\text{t}C=P_\gamma C
\end{equation}
for every $\gamma\in\CN$ and
\begin{equation}\label{fourier}
\gamma(Q)^\text{t}\overline{\delta(Q)}=\begin{cases}
CP_{\gamma^{-1}\delta}&\text{if }\gamma\equiv\delta\pmod{\mathcal{N}}\\
0&\text{otherwise}
\end{cases}
\end{equation}
for $\gamma,\delta\in\mathcal{G}$. For any subset $\CS\subseteq\CN$ we write $P_\CS:=\sum_{\delta\in\CS}P_\delta$.

\begin{Lem}\label{WPPW}
In the situation of Theorem~B we may assume that $W$ is (integral) positive definite and commutes with $P_\gamma$ for every $\gamma\in\CN$.
\end{Lem}
\begin{proof}
Let 
\[\CW:=\frac{1}{2n}\sum_{\delta\in\CN}P_\delta(W+W^\text{t})P_\delta^\text{t}.\] 
Then $\CW$ is symmetric and commutes with $P_\delta$ for every $\delta\in\CN$. Moreover, $\CW$ is integral positive definite and by \autoref{lemsym}, $\CW$ is positive definite. Finally, 
\[\tr(\CW C)=\frac{1}{2n}\sum_{\delta\in\CN}\tr(P_\delta WCP_\delta^\text{t})+\tr(P_\delta W^\text{t}CP_\delta^\text{t})=\tr(WC),\]
since $P_\delta$ commutes with $C$. Hence, we may replace $W$ by $\CW$.
\end{proof}

In the following we revisit some arguments from \cite[Section~5.2]{habil}.
Write $Q=\sum_{i=1}^{\phi(q)}A_i\zeta^i$ where $A_i\in\ZZ^{k\times l}$ for $i=1,\ldots,\phi(q)$ and $\phi(q)=q-q/p$ is Euler's function. Let \[\CA_q=\bigl(A_i:i=1,\ldots,\phi(q)\bigr)\in\ZZ^{k\times\phi(q)l}.\]

\begin{Lem}\label{rank}
The matrix $\CA_q$ has rank $l\phi(q)/n$.
\end{Lem}
\begin{proof}
It is well-known that the Vandermonde matrix $V:=(\zeta^{i\gamma}:1\le i\le\phi(q),\gamma\in\CG)$ is invertible. 
Since $Q$ has full rank, the facts stated above show that $(\gamma(Q):\gamma\in\CG)$ has rank $l|\CG:\CN|=l\phi(q)/n$. Then also  $\CA_q=(\gamma(Q):\gamma\in\CG)(V\otimes 1_l)^{-1}$ has rank $l\phi(q)/n$. 
\end{proof}

Let $T_q$ be the trace of $\QQ(\zeta)$ with respect to $\QQ$. Recall that
\[T_q(\zeta^i)=\begin{cases}
\phi(q)&\text{if }q\mid i,\\
-q/p&\text{if }q\nmid i\text{ and } \frac{q}{p}\mid i,\\
0&\text{otherwise}.
\end{cases}\]
Hence,
\[T_q(Q\zeta^{-i})=\sum_{j=1}^{\phi(q)}A_jT_q(\zeta^{j-i})=\frac{q}{p}\Bigl(pA_i-\sum_{j\equiv i\pmod{q/p}}A_j\Bigr).\]

\begin{Def}
For $1\le i\le\phi(q)$ there is a unique integer $i'$ such that $0\le i'<q/p$ and $i'\equiv-i\pmod{q/p}$.
\end{Def}

Then $q/p\le i+i'\le\phi(q)$ and $\sum_{j\equiv i\pmod{q/p}}\zeta^{-j}=-\zeta^{i'}$ where we consider only those summands with $1\le j\le\phi(q)$. 
With this convention we obtain
\begin{align*}
T_q\bigl(Q(\zeta^{-i}-\zeta^{i'})\bigr)&=\frac{q}{p}\Bigl(pA_i-\sum_{j\equiv i\pmod{q/p}}A_j+\sum_{j\equiv i\pmod{q/p}}\bigl(pA_j-\sum_{s\equiv j\pmod{q/p}}A_s\bigr)\Bigr)\\
&=\frac{q}{p}\Bigl(pA_i+(p-1)\sum_{j\equiv i\pmod{q/p}}A_j-(p-1)\sum_{s\equiv i\pmod{q/p}}A_s\Bigr)=qA_i
\end{align*}
and \eqref{fourier} yields
\begin{align*}
q^2A_i^\text{t}A_j&=\sum_{\gamma,\delta\in\mathcal{G}}(\zeta^{-i\gamma}-\zeta^{i'\gamma})(\zeta^{j\delta}-\zeta^{-j'\delta})\gamma(Q)^\text{t}\delta(\overline{Q})\\
&=\sum_{\delta\in\mathcal{N}}\sum_{\gamma\in\mathcal{G}}(\zeta^{-i\gamma}-\zeta^{i'\gamma})(\zeta^{j\gamma\delta}-\zeta^{-j'\gamma\delta})qCP_{\delta}\\
&=qC\sum_{\delta\in\mathcal{N}}P_{\delta}T_q\bigl(\zeta^{j\delta-i}-\zeta^{j\delta+i'}-\zeta^{-j'\delta-i}+\zeta^{-j'\delta+i'}\bigr)
\end{align*}
for $1\le i,j\le\phi(q)$. Note that
\[j\delta-i\equiv j\delta+i'\equiv -j'\delta-i\equiv -j'\delta+i'\pmod{q/p}.\]
Moreover, if $j\delta-i\equiv 0\pmod{q}$, then $j\delta-i'\not\equiv 0\pmod{p}$. In this case $T_q(\zeta^{j\delta-i}-\zeta^{j\delta+i'})=\phi(q)+q/p=q$. In a similar way we obtain
\begin{equation}\label{longeq}
\boxed{A_i^\text{t}A_j=C\sum_{\delta\in\CN}P_\delta\bigl([j\delta\equiv i]-[j\delta\equiv -i']+[j'\delta\equiv i']-[j'\delta\equiv-i]\bigr)}
\end{equation}
where all congruences are modulo $q$ and $[\ldots]$ denotes the indicator function. 

By \autoref{padic}, $\CA_q$ has non-zero rows $a_1,\ldots,a_{k_0(B)}$. If $\CW\in\RR^{l\phi(q)\times l\phi(q)}$ is integral positive definite, then
\[k_0(B)\le\sum_{i=1}^{k_0(B)}a_i\CW a_i^\text{t}\le\tr(\CA_q\CW\CA_q^\text{t})=\tr(\CW\CA_q^\text{t}\CA_q)\]
and this is what we are going to show.
We need to discuss the case $p=2$ separately.

\begin{Prop}\label{p2}
Theorem~B holds for $p=2$.
\end{Prop}
\begin{proof}
If $q\le 2$, then $Q=A_1=\CA_q$, $n=1$ and
\[k_0(B)\le \tr(WQ^\text{t}Q)=\tr(WqC)=q\tr(WC)=k_0(\langle u\rangle\rtimes\CN)\tr(WC).\] 
Hence, we will assume for the remainder of the proof that $q\ge 4$.
Then $i'=q/2-i$ for every $1\le i\le\phi(q)=q/2$. Hence, \eqref{longeq} simplifies to
\begin{equation}\label{short}
A_i^\text{t}A_j=2C\sum_{\delta\in\CN}P_\delta\bigl([j\delta\equiv i]-[j\delta\equiv i+q/2]\bigr).
\end{equation}
It is well-known that 
\[\CG=\langle -1+q\ZZ\rangle\times\langle 5+q\ZZ\rangle\cong C_2\times C_{q/4}.\] 
In particular, $\CN$ is a $2$-group and so is $U:=\langle u\rangle\rtimes\CN$. Therefore, $k_0(U)=|U:U'|$ where $U'$ denotes the commutator subgroup of $U$.

\textbf{Case~1:} $\CN=\langle 5^{2^m}+q\ZZ\rangle$ for some $m\ge 0$.\\
Then $q=|\CN|2^{m+2}=n2^{m+2}$ and $U'$ is generated by $u^{5^{2^m}-1}$. Since $5^{2^m}-1\equiv 2^{m+2}\pmod{2^{m+3}}$, we conclude that $|U'|=n$ and $k_0(U)=|U:U'|=q$. 

For any given $\delta\in\CN\setminus\{1\}$ both congruences $i\delta\equiv i\pmod{q}$ and $i\delta\equiv i+q/2\pmod{q}$ have solutions $i\in\{1,\ldots,q/2\}$. Moreover, the number of solutions is the same, since they both form residue classes modulo a common integer. On the other hand, $i\delta\equiv i+q/2\pmod{q}$ has no solution for $\delta=1$.
An application of \eqref{short} yields
\[\sum_{i=1}^{q/2}A_i^\text{t}A_i=2C\sum_{\delta\in\CN}P_\delta\sum_{i=1}^{q/2}[i\delta\equiv i]-[i\delta\equiv i+q/2]=qCP_1=qC.\]
The matrix $\CW:=1_{q/2}\otimes W$ is certainly integral positive definite. Moreover, 
\begin{equation}\label{52n}
k_0(B)\le\tr(\CW\CA_q^\text{t}\CA_q)=\tr\Bigl(\sum_{i=1}^{q/2}WA_i^\text{t}A_i\Bigr)=q\tr(WC)=k_0(U)\tr(WC).
\end{equation}
It remains to check when this bound is sharp. If $k_0(B)=\tr(\CW\CA_q^\text{t}\CA_q)$, then every row of $\CA_q$ vanishes in all but (possibly) one $A_i$. Moreover, characters of positive height vanish completely in $\CA_q$. 
By way of contradiction, suppose that $\CN$ acts non-trivially on $\IBr(b)$. Using \eqref{relation}, it follows that there exists a character $\chi\in\Irr(B)$ of height $0$ such that the corresponding row $d_\chi=a\zeta^i$ of $Q$ satisfies $aP_\delta=-a$ for some $\delta\in\CN$. We write $a=(\alpha_1,\ldots,\alpha_s,-\alpha_1,\ldots,-\alpha_s,0,\ldots,0)$ with non-zero $\alpha_1,\ldots,\alpha_s\in\ZZ$. With the notation of \autoref{padic} let $\widetilde{C}=(\widetilde{c}_{ij})$. By \eqref{CPPC}, we have $P_\delta\widetilde{C}=\widetilde{C}P_\delta$. Now \autoref{padic} leads to the contradiction
\[0\not\equiv d_\chi\widetilde{C}\overline{d_\chi}^\text{t}=a\widetilde{C}a^\text{t}\equiv\sum_{i=1}^s2\alpha_i^2\widetilde{c}_{ii}\equiv 0\pmod{2},\]
since the diagonal of $\widetilde{C}$ is constant on the orbits of $\CN$.
Therefore, equality in \eqref{52n} can only hold if $\CN$ acts trivially on $\IBr(b)$.

\textbf{Case~2:} $\delta:=-5^m+q\ZZ\in\CN\setminus\{1\}$ for some $m\ge 0$.\\
Since $1+5^m\equiv 2\pmod{4}$, we have $U'=\langle u^{1+5^m}\rangle=\langle u^2\rangle$ and $k_0(U)=|U:U'|=2n$. 
We show that every row of $A_{q/2}$ corresponding to a height $0$ character $\chi\in\Irr(B)$ is non-zero.
Let $d_\chi=\sum_{i=1}^{q/2}a_i\zeta^i$ be the corresponding row of $Q$ where $a_i$ is a row of $A_i$. Let $\nu$ be the $p$-adic valuation. By \autoref{padic}, 
\begin{align*}
0&=\nu(d_\chi\widetilde{C}\overline{d_\chi}^\text{t})=\nu\Bigl(\sum_{1\le i,j\le q/2}a_i\widetilde{C}a_j^\text{t}\zeta^{i-j}\Bigr)=\nu\Bigl(\sum_{i=1}^{q/2}a_i\widetilde{C}a_i^\text{t}\Bigr),
\end{align*}
i.\,e. 
\begin{equation}\label{nu}
\sum_{i=1}^{q/2}a_i\widetilde{C}a_i^\text{t}\equiv 1\pmod{2}.
\end{equation}
On the other hand, 
\[\sum_{i=1}^{q/2}a_iP_\delta \zeta^i=d_\chi P_\delta=\delta(d_\chi)=\sum_{i=1}^{q/2}a_i\zeta^{i\delta}.\]
Now $i\delta\equiv i\pmod{q}$ implies $-5^m\equiv\delta\equiv 1\pmod{q/\gcd(q,i)}$ and $i=q/2$. Similarly $i\delta\equiv i+q/2\pmod{q}$ implies $i=q/4$. Then $A_{q/4}P_\delta=-A_{q/4}$. As in Case~1, it follows that $a_{q/4}\widetilde{C}a_{q/4}^\text{t}\equiv 0\pmod{2}$.
For $i\notin\{q/2,q/4\}$ we have $A_iP_\delta=\pm A_j$ for some $j\in\{1,\ldots,q/2\}\setminus\{i\}$.
Then, using \eqref{CPPC},
\[a_j\widetilde{C}a_j^\text{t}=a_iP_\delta\widetilde{C}P_\delta^\text{t}a_i^\text{t}=a_i\widetilde{C}a_i^\text{t}.\]
Now \eqref{nu} yields $a_{q/2}\widetilde{C}a_{q/2}^\text{t}\equiv 1\pmod{2}$ and $a_{q/2}\ne 0$. Therefore, $A_{q/2}$ has non-zero rows for height $0$ characters. 

By \eqref{short}, $A_{d/2}^\text{t}A_{d/2}=2CP_\CN$ and \autoref{lemposperm} implies
\[k_0(B)\le\tr(WA_{d/2}^\text{t}A_{d/2})=2\tr(WCP_\CN)=2\sum_{\gamma\in\CN}\tr(WCP_\gamma)\le 2n\tr(WC)=k_0(U)\tr(WC)\]
with strict inequality if $\CN$ acts non-trivially on $\IBr(b)$.
\end{proof}

We are left with the case $p>2$. 
Here $\CG$ is cyclic and $\CN$ is uniquely determined by $n$. Let $n_p$ be the $p$-part of $n$ and $n_{p'}$ the $p'$-part. Then $n_p\mid \frac{q}{p}$ and $n_{p'}\mid p-1$.

\begin{Lem}\label{k0}
We have $k_0\bigl(\langle u\rangle\rtimes\CN\bigr)=n+\frac{q-n_p}{n_{p'}}$ for $p>2$.
\end{Lem}
\begin{proof}
The inflations from $\mathcal{N}$ yield $n$ linear characters in $U:=\langle u\rangle\rtimes\mathcal{N}$, since $\mathcal{N}$ is cyclic. Now let $1\ne\lambda\in\Irr(\langle u\rangle)$. If the orbit size of $\lambda$ under $\mathcal{N}$ is divisible by $p$, then the irreducible characters of $U$ lying over $\lambda$ all have positive height. Hence, we may assume that $\lambda^{q/n_p}=1$. 
Then, by Clifford theory, $\lambda$ extends in $n_p$ many ways to $\langle u\rangle\rtimes\mathcal{N}_p$ where $\mathcal{N}_p$ is the Sylow $p$-subgroup of $\mathcal{N}$. All these extensions induce to irreducible characters of $U$ of height $0$. We have $\frac{q/n_p-1}{n_{p'}}$ choices for $\lambda$. Thus, in total we obtain 
\[k_0(U)=n+n_p\frac{q/n_p-1}{n_{p'}}=n+\frac{q-n_p}{n_{p'}}.\qedhere\]
\end{proof}

The following settles Theorem~B in the special case $n_p=1$ (use \autoref{k0}).

\begin{Prop}\label{prop}
Let $p>2$ and $n_p=1$.
With the notation above there exists an integral positive definite matrix $\CW\in\RR^{\phi(q)l\times \phi(q)l}$ such that 
\[\tr(\CW\CA_q^\text{t}\CA_q)\le \bigl(n+\frac{q-1}{n}\bigr)\tr(WC)\] 
with equality if and only if $\CN$ acts trivially on $\IBr(b)$.
\end{Prop}

\begin{proof}
We argue by induction on $q$. If $q=1$, then $\CA_q=A_1=Q$, $n=1$ and the claim holds with $\CW=W$ (\autoref{WPPW}). 
The next case requires special treatment as well.

\textbf{Case~1:} $q=p$.\\
Then $i'=0$ for all $i$ and \eqref{longeq} simplifies to
\[
A_i^\text{t}A_j=C\sum_{\delta\in\CN}P_\delta\bigl([j\delta\equiv i]+[0\delta\equiv0]\bigr)=
\begin{cases}
CP_\CN&\text{if }i\not\equiv j\pmod{\mathcal{N}},\\
C(P_\CN+P_{j^{-1}i})&\text{if }i\equiv j\pmod{\mathcal{N}}.
\end{cases}
\]
After permuting the columns of $\CA_q$ if necessary, we obtain
\[\CA_q^\text{t}\CA_q=1^{\phi(q)\times\phi(q)}\otimes P_\CN C+1_{n'}\otimes(P_{\gamma^{-1}\delta}C)_{\gamma,\delta\in\mathcal{N}}\]
where $n':=(p-1)/n$. We fix a generator $\rho$ of $\CN$. Then we may write $(P_{\gamma^{-1}\delta}C)_{\gamma,\delta\in\mathcal{N}}=(P_\rho^{j-i}C)_{i,j=1}^n$. 

By \autoref{WPPW}, we may assume that $W$ is (integral) positive definite and commutes with $P_\rho$. 
Let $W_n$ as in \autoref{lemwn} where we use $P_\rho$ instead of $P$. A repeated application of that lemma shows that the matrix $\CW:=U_{n'}\otimes W_n$ is integral positive definite.
Moreover, since $P_\CN P_\rho=P_\CN=P_\CN P_\rho^\text{t}$, we have
\begin{align*}
\tr(\CW \CA_q^\text{t}\CA_q)&=\tr\bigl((U_{n'}\otimes W_n)(1^{\phi(q)\times\phi(q)}\otimes P_\CN C)\bigr)+\tr\bigl((U_{n'}\otimes W_n)(1_{n'}\otimes(P_\rho^{j-i}C))\bigr)\\
&=\tr\bigl((U_{n'}\otimes W_n)(1^{n'\times n'}\otimes 1^{n\times n}\otimes P_\CN C)\bigr)+\tr\bigl(U_{n'}\otimes W_n(P_\rho^{j-i}C)\bigr)\\
&=\tr\bigl(U_{n'}1^{n'\times n'}\bigr)\tr\bigl(W_n(1^{n\times n}\otimes P_\CN C)\bigr)+\tr(U_{n'})\tr\bigl(W_n(P_\rho^{j-i}C)\bigr)\\
&=\tr\bigl(W_n(1^{n\times n}\otimes P_\CN C)\bigr)+n'\tr\bigl(W_n(P_\rho^{j-i}C)\bigr)\\
&=\sum_{i=1}^n\tr(WCP_\CN )-\sum_{i=1}^{n-1}\tr(WCP_\CN P_\rho)+n'\Bigl(\sum_{i=1}^n\tr(WC)-\sum_{i=1}^{n-1}\tr(WCP_\rho P_\rho^\text{t})\Bigr)\\
&=\tr(WCP_\CN)+n'\tr(WC).
\end{align*}
Finally, \autoref{lemposperm} implies
\[\tr(WCP_\CN)=\sum_{\delta\in\CN}\tr(WCP_\delta)\le n\tr(WC)\]
with equality if and only if $\CN$ acts trivially on $\IBr(b)$. This completes the proof in the case $q=p$.

\textbf{Case~2:} $q>p$.\\
Let 
\begin{align*}
I_p&:=\{1\le i\le\phi(q):p\mid i\},&
I_{p'}&:=\{1\le i\le\phi(q):p\nmid i\}.
\end{align*}
Then $|I_p|=\phi(q)/p=\phi(q/p)$ and $|I_{p'}|=\phi(q)-\phi(q/p)=\phi(q/p)(p-1)$.
If $i\in I_p$ and $j\in I_{p'}$, then $j\delta-i\not\equiv 0\pmod{q/p}$ for every $\delta\in\CN$ and $A_i^\text{t}A_j=0$ by \eqref{longeq}. 
Hence, after relabeling the columns of $\CA_q$, we obtain
\[\CA_q^\text{t}\CA_q=\begin{pmatrix}
\Delta_p&0\\0&\Delta_{p'}
\end{pmatrix}\]
where $\Delta_p$ corresponds to the indices in $I_p$.
Since $n\mid p-1$, we may regard $\CN$ as a subgroup of $\Gal(\QQ(\zeta^p)|\QQ)$.
For $i\in I_p$ let $j=i/p$. Then $i'\equiv -i\pmod{q/p}$ implies $i'/p\equiv -j\pmod{q/p^2}$ and $0\le i'/p<q/p^2$. Hence, $j'=i'/p$ where the left hand side refers to $q/p$. It follows from \eqref{longeq} that
$\Delta_p=\CA_{q/p}^\text{t}\CA_{q/p}$. By induction on $q$ there exists an integral positive definite $\CW_p$ such that \[\tr(\CW_p\Delta_p)\le \Bigl(n+\frac{q/p-1}{n}\Bigr)\tr(WC)\] 
with equality if and only if $\CN$ acts trivially on $\IBr(b)$.

It remains to consider $\Delta_{p'}$. By \autoref{rank}, $\CA_{q/p}$ and $\Delta_p$ have rank $l\phi(q/p)/n$ and therefore $\Delta_{p'}$ has rank 
\[l(\phi(q)-\phi(q/p))/n=l\phi(q/p)(p-1)/n.\] 
We define a subset $J\subseteq I_{p'}$ such that $|J|=\phi(q/p)(p-1)/n$ and the matrix $(A_i:i\in J)$ has full rank. Let $R$ be a set of representatives for the orbits of $\{i\in I_{p'}:1\le i\le q/p\}$ under the multiplication action of $\CN$ modulo $q/p$. Note that every orbit has size $n$. For $r\in R$ let 
\[J_r:=\{r+jq/p:j=0,\ldots,p-2\}\subseteq I_{p'}\] 
and $J:=\bigcup_{r\in R}J_r$. Since $J_r\cap J_s=\varnothing$ for $r\ne s$, we have $|J|=\phi(q/p)(p-1)/n$. If $i\in J_r$ and $j\in J_s$ with $r\ne s$, then $j\delta\not\equiv i\pmod{q/p}$ for every $\delta\in\CN$. Consequently, $A_i^\text{t}A_j=0$. Now let $i,j\in J_r$. Then \eqref{longeq} implies
\[A_i^\text{t}A_j=C(1+\delta_{ij}).\]
After relabeling we obtain
\[(A_i:i\in J)^\text{t}(A_i:i\in J)=1_{\phi(q/p)/n}\otimes (1+\delta_{ij})_{i,j=1}^{p-1}\otimes C.\]
In particular, $(A_i:i\in J)$ has full rank. Since $\Delta_{p'}$ has the same rank, there exists an integral matrix 
$S\in\GL(l\phi(q/p)(p-1),\QQ)$ such that %$(A_i:i\in I_{p'})=(A_i:j\in J)
\[S^\text{t}\Delta_{p'}S=1_{\phi(q/p)/n}\otimes (1+\delta_{ij})\otimes C\oplus 0_s\]
where $s:=l\phi(q/p)(p-1)(n-1)/n$. 
Let 
\[\CW_{p'}:=S\bigl(1_{\phi(q/p)/n}\otimes U_{p-1}\otimes W\oplus 1_s\bigr)S^{\text{t}}.\]
Then $\CW_{p'}$ is integral positive definite by \autoref{lemwn}.
Moreover, 
\begin{align*}
\tr(\CW_{p'}\Delta_{p'})&=\tr\bigl((1_{\phi(q/p)/n}\otimes U_{p-1}\otimes W)(1_{\phi(q/p)/n}\otimes (1+\delta_{ij})\otimes C)\bigr)+\tr(1_s0_s)\\
&=\frac{\phi(q/p)}{n}\tr\bigl(U_{p-1}(1+\delta_{ij})\bigr)\tr(WC)=\frac{\phi(q/p)p}{n}\tr(WC)=\frac{\phi(q)}{n}\tr(WC).
\end{align*}
Finally, we set $\CW:=\CW_p\oplus\CW_{p'}$. Then $\CW$ is integral positive definite and
\begin{align*}
\tr(\CW\CA_q^\text{t}\CA_q)&=\tr(\CW_p\Delta_p)+\tr(\CW_{p'}\Delta_{p'})\le \Bigl(n+\frac{q/p-1}{n}\Bigr)\tr(WC)+\frac{\phi(q)}{n}\tr(WC)\\
&=\Bigl(n+\frac{q-1}{n}\Bigr)\tr(WC)
\end{align*}
with equality if and only if $\CN$ acts trivially on $\IBr(b)$. 
\end{proof}

To complete the proof of Theorem~B it remains to show the following.

\begin{Prop}
Theorem~B holds in the case $p>2$ and $n_p>1$.
\end{Prop}
\begin{proof}
Let 
\begin{align*}
I_1:=\{1\le i\le\phi(q):n_p\mid i\},&&I_2:=\{1\le i\le\phi(q):n_p\nmid i\}.
\end{align*}
As in the proof of \autoref{prop} we have
\[\CA_q^\text{t}\CA_q=\begin{pmatrix}
\Delta_1&0\\0&\Delta_2
\end{pmatrix}\]
where $\Delta_1$ corresponds to the indices in $I_1$. Let $\CN=\CN_p\times\CN_{p'}$ where $\CN_p:=\langle 1+q/n_p+q\ZZ\rangle$ is the unique Sylow $p$-subgroup of $\CN$. Then $\delta i\equiv i\pmod{q}$ for $\delta\in\CN_p$ and $i\in I_1$. Hence, for $i,j\in I_1$ we have
\begin{align*}
A_i^\text{t}A_j&=C\sum_{\delta\in\CN}P_\delta\bigl([j\delta\equiv i]-[j\delta\equiv -i']+[j'\delta\equiv i']-[j'\delta\equiv-i]\bigr)\\
&=CP_{\CN_p}\sum_{\delta\in\CN_{p'}}P_\delta\bigl([j\delta\equiv i]-[j\delta\equiv -i']+[j'\delta\equiv i']-[j'\delta\equiv-i]\bigr).
\end{align*}
For $i\in I_1$ it is easy to see that $i'/n_p=(i/n_p)'$ when the right hand side is considered with respect to $q/n_p$ (see proof of \autoref{prop}).
It follows that
\[\Delta_1=(1_{\phi(q/n_p)}\otimes P_{\CN_p})\CA_{q/n_p}^\text{t}\CA_{q/n_p}\]
where we consider $\CA_{q/n_p}$ with respect to the $p'$-group $\CN_{p'}$.
By \autoref{prop}, there exists an integral positive definite $\CW_1$ such that
\begin{equation}\label{delta1}
\tr(\CW_1\CA_{q/n_p}^\text{t}\CA_{q/n_p})\le \Bigl(n_{p'}+\frac{q/n_p-1}{n_{p'}}\Bigr)\tr(WC).
\end{equation} 
Moreover, equality holds if and only if $\CN_{p'}$ acts trivially on $\IBr(b)$.
By construction, $\CA_{q/n_p}^\text{t}\CA_{q/n_p}$ is positive semidefinite. 
By \eqref{CPPC} and \eqref{longeq}, $\CA_{q/n_p}^\text{t}\CA_{q/n_p}$ commutes with $1_{\phi(q/n_p)}\otimes P_{\delta}$ for $\delta\in\CN_p$. Hence, \autoref{lemposperm} implies
\begin{equation}\label{pmidn}
\begin{split}
\tr(\CW_1\Delta_1)&=\tr\bigl(\CW_1(1_{\phi(q/n_p)}\otimes P_{\CN_p})\CA_{q/n_p}^\text{t}\CA_{q/n_p}\bigr)\\
&\le n_p\tr(\CW_1\CA_{q/n_p}^\text{t}\CA_{q/n_p})\le\Bigl(n+\frac{q-n_p}{n_{p'}}\Bigr)\tr(WC).
\end{split}
\end{equation}
Suppose that $\tr(\CW_1\Delta_1)=\bigl(n+\frac{q-n_p}{n_{p'}}\bigr)\tr(WC)$. Then, by \eqref{delta1}, $\CN_{p'}$ acts trivially on $\IBr(b)$ and the matrices $A_i^\text{t}A_j$ with $i,j\in I_1$ are scalar multiples of $CP_{\CN_p}$. 
We write $\CA_{q/n_p}^\text{t}\CA_{q/n_p}=(A_{ij})$ such that $A_{in_p}^\text{t}A_{jn_p}=P_{\CN_p}A_{ij}$.
Note that $A_{11}=2C$ is positive definite. 
As in the proof of \autoref{lemposperm}, we construct a positive semidefinite matrix $M=(m_{ij}):=A^{1/2}\CW_1A^{1/2}$ where $A^{1/2}A^{1/2}=(A_{ij})_{i,j}$. By way of contradiction, suppose that $P_\delta\ne 1_l$ for some $\delta\in\CN_p$. Let $1\le i\le l$ such that $\delta(i)\ne i$, and let $x=(x_j)\in\ZZ^{\phi(q/n_p)l}$ with $x_i=-x_{\delta(i)}=1$ and zero elsewhere.
Then $x(A_{ij})x^\text{t}>0$ since $A_{11}$ is positive definite. Thus, $A^{1/2}x^\text{t}\ne 0$. Since $\CW_1$ is positive definite (\autoref{lemsym}), it follows that $xMx^\text{t}>0$ and $m_{i\delta(i)}<(m_{ii}+m_{\delta(i)\delta(i)})/2$. Hence, the proof of \autoref{lemposperm} leads to 
\begin{align*}
\tr\bigl(\CW_1(1_{\phi(q/n_p)}\otimes P_\delta)\CA_{q/n_p}^\text{t}\CA_{q/n_p}\bigr)&=\tr(A^{1/2}\CW_1(1_{\phi(q/n_p)}\otimes P_\delta)A^{1/2})\\
&=\tr(M(1_{\phi(q/n_p)}\otimes P_\delta))<\tr(M)=\tr(\CW_1\CA_{q/n_p}^\text{t}\CA_{q/n_p})
\end{align*}
and we derive the contradiction $\tr(\CW_1\Delta_1)<n_p\tr(\CW_1\CA_{q/n_p}^\text{t}\CA_{q/n_p})$. Thus, we have shown that equality in \eqref{pmidn} can only hold if $\CN$ acts trivially on $\IBr(b)$.

Now we use the argument from \autoref{p2} to deal with $\Delta_2$. Let $\chi\in\Irr(B)$ of height $0$, and let $d_\chi=\sum_{i=1}^{\phi(q)}a_i\zeta^i$ be the corresponding row of $Q$. By \autoref{padic}, we have
\[0=\nu(d_\chi\widetilde{C}\overline{d_\chi}^\text{t})=\nu\Bigl(\sum_{i,j=1}^{\phi(q)}a_i\widetilde{C}a_j^\text{t}\Bigr)\]
where $\nu$ is the $p$-adic valuation. 
In order to show that $a_i\ne 0$ for some $i\in I_1$, it suffices to show that
\begin{equation}\label{widec}
\sum_{i,j\in I_2}a_i\widetilde{C}a_j^\text{t}\equiv 0\pmod{p}.
\end{equation}
For any $\delta\in\CN_p$ we have
\[\sum_{i=1}^{\phi(q)}A_iP_\delta\zeta^i=QP_\delta=\delta(Q)=\sum_{i=1}^{\phi(q)}A_i\zeta^{i\delta}.\]
Restricting to the indices $i\in I_2$ and taking the valuation yields
\[\sum_{i\in I_2}A_iP_\delta\equiv\sum_{i\in I_2}A_i\pmod{p}.\]
Let $i\in I_2$ be arbitrary and choose $\delta\in\CN_p$ such that $\gcd(q,i)p=|\langle\delta\rangle|$. Let
\[\{i_1,\ldots,i_{p-1}\}=\bigl\{j\in I_2:j\equiv i\pmod{q/p}\bigr\}.\] 
We may assume that $i_1\delta\equiv-i'\pmod{q}$ and $i_j\delta\equiv i_{j-1}\pmod{q}$ for $j=2,\ldots,p-1$.
Since $\zeta^{-i'}=-\zeta^{i_1}-\ldots-\zeta^{i_{p-1}}$, we obtain $A_{i_{p-1}}P_\delta=-A_{i_1}$ and $A_{i_j}P_\delta=A_{i_{j+1}}-A_{i_1}$ for $j=1,\ldots,p-2$. 
Hence, 
\begin{align*}
\Bigl(\sum_{j\in I_2}a_j\Bigr)\widetilde{C}a_{i_1}^\text{t}&=\Bigl(\sum_{j\in I_2}a_j\Bigr)P_\delta\widetilde{C}P_\delta^\text{t}a_{i_1}^\text{t}\equiv \Bigl(\sum_{j\in I_2}a_j\Bigr)\widetilde{C}(a_{i_2}-a_{i_1})^\text{t}\equiv \Bigl(\sum_{j\in I_2}a_j\Bigr)\widetilde{C}(a_{i_3}-a_{i_2})^\text{t}\\
&\equiv\ldots\equiv \Bigl(\sum_{j\in I_2}a_j\Bigr)\widetilde{C}(a_{i_{p-1}}-a_{i_{p-2}})^\text{t}\equiv -\Bigl(\sum_{j\in I_2}a_j\Bigr)\widetilde{C}a_{i_{p-1}}^\text{t}\pmod{p}.
\end{align*}
Now it is easy to see that
\[\Bigl(\sum_{j\in I_2}a_j\Bigr)\widetilde{C}(a_{i_1}+\ldots+a_{i_{p-1}})^\text{t}\equiv \frac{p(p-1)}{2}\Bigl(\sum_{j\in I_2}a_j\Bigr)\widetilde{C}a_{i_1}^\text{t}\equiv 0\pmod{p}\]
and \eqref{widec} follows. Thus, we have shown that every height $0$ character has a non-vanishing part in $A_i$ for some $i\in I_1$. 
Hence by \eqref{pmidn}, 
\[k_0(B)\le\tr(\CW_1\Delta_1)\le \Bigl(n+\frac{q-n_p}{n_{p'}}\Bigr)\tr(WC)\]
with strict inequality if $\CN$ acts non-trivially on $\IBr(b)$. By \autoref{k0}, the proof is complete.
\end{proof}

Now it is time to derive Theorem~A from Theorem~B. For the convenience of the reader we restate it as follows.

\begin{Prop}
If $u\in\Z(D)$ in the situation above, then
\[k(B)\le \Bigl(n+\frac{q-1}{n}\Bigr)\tr(WC)\le q\tr(WC).\]
The first inequality is strict if $\mathcal{N}$ acts non-trivially on $\IBr(b)$ and the second inequality is strict if and only if $1<n<q-1$.
\end{Prop}
\begin{proof}
As mentioned in the introduction, $\CN$ is induced from the inertial quotient $\N_G(D,b_D)/D\C_G(D)$ and therefore $\CN$ is a $p'$-group. As a subgroup of $\Aut(\langle u\rangle)$, its order $n$ must divide $p-1$. 
For $p=2$ we obtain $n=1$ and $k_0(\langle u\rangle\rtimes\CN)=q$. For $p>2$, \autoref{k0} gives $k_0(\langle u\rangle\rtimes\CN)=n+\frac{q-1}{n}$. 
By \autoref{padic}, all rows of $Q$ are non-zero. Hence, the proofs of Propositions \ref{p2} and \ref{prop} actually show that $k(B)\le k_0(\langle u\rangle\rtimes\CN)\tr(WC)$ with strict inequality if $\CN$ acts non-trivially on $\IBr(b)$ (note that only Case~1 in the proof of \autoref{p2} is relevant). 
This implies the first two claims. The last claim follows, since $n+\frac{q-1}{n}$ is a convex function in $n$ and $1\le n\le q-1$.
\end{proof}

If the action of $\CN$ on $\IBr(b)$ is known, a careful analysis of the proofs above leads to even stronger estimates. For instance, in \autoref{prop} we have actually shown that
\[k_0(B)\le\tr(WCP_\CN)+\frac{q-1}{n}\tr(WC)\]
for $p>2$ and $n_p=1$. If $\overline{b}$ has cyclic defect groups, then $P_\CN$ is a direct sum of equal blocks of the form $d^{n/d\times n/d}$ (see \cite[Proposition~3.2]{Sambalerefine}). This can be used to give a simpler proof of \cite[Theorem~3.1]{Sambalerefine}.

\section{Consequences}

In this section we deduce some of the results stated in the introduction. 

\begin{Cor}[{Sambale~\cite[Lemma~1]{SambalekB}}]\label{sam}
Let $C=(c_{ij})_{i,j=1}^l$ be the Cartan matrix of a Brauer correspondent of $B$ in $\C_G(u)$ where $u\in\Z(D)$.
Then for every positive definite, integral quadratic form $q(x_1,\ldots,x_l)=\sum_{1\le i\le j\le l}{q_{ij}x_ix_j}$ we have 
\[k(B)\le\sum_{1\le i\le j\le l}{q_{ij}c_{ij}}.\]
\end{Cor}
\begin{proof}
Let $t:=|\langle u\rangle|$. Then $t^{-1}C$ is the Cartan matrix of the block $\overline{b}$ in Theorem~A (see \cite[Theorem~1.22]{habil}).
Taking $W:=\frac{1}{2}(q_{ij}(1+\delta_{ij}))$ with $q_{ij}=q_{ji}$ we obtain 
\[xWx^\text{t}=\frac{1}{2}\sum_{1\le i,j\le l}{q_{ij}(1+\delta_{ij})x_ix_j}=\sum_{1\le i\le j\le l}{q_{ij}x_ix_j}=q(x)\ge 1\]
for every $x=(x_1,\ldots,x_l)\in\ZZ^l\setminus\{0\}$ and 
\[k(B)\le t\tr(Wt^{-1}C)=\tr(WC)=\sum_{1\le i\le j\le l}{q_{ij}c_{ij}}.\qedhere\]
\end{proof}

Wada's inequality~\eqref{W} follows from \autoref{sam} with $q(x)=\sum_{i=1}^lx_i^2-\sum_{i=1}^{l-1}x_ix_{i+1}$ (or $W=U_l$ in Theorem~A). 

\begin{Cor}[{Héthelyi--Külshammer--Sambale~\cite[Theorem~4.10]{HKS}}]
Suppose $p>2$.
Let $b$ be a Brauer correspondent of $B$ in $\C_G(u)$ where $u\in D$ and $l(b)=1$. Let $\lvert\N_G(\langle u\rangle,b):\C_G(u)\rvert=p^sr$ with $s\ge 0$ and $p\nmid r$. Then
\[k_0(B)\le\frac{|\langle u\rangle|+p^s(r^2-1)}{|\langle u\rangle|r}p^d\]
where $d$ is the defect of $b$.
\end{Cor}
\begin{proof}
Setting $q:=|\langle u\rangle|$ we obtain $C=(p^d/q)$ in the situation of Theorem~B. By \autoref{k0}, $k_0(\langle u\rangle\rtimes\CN)=(q+p^s(r^2-1))/r$ and the claim follows with $W=(1)$.
\end{proof}

The following result of Brauer cannot be seen in the framework of integral quadratic forms. It was a crucial ingredient in the proof of the $k(GV)$-Problem (see \cite[Theorem~2.5d]{Schmid}).

\begin{Cor}[{Brauer~\cite[5D]{BrauerBlSec2}}]\label{brauer}
Let $B$ be a $p$-block with defect $d$, and let $C$ be the Cartan matrix of a Brauer correspondent $b$ of $B$ in $\C_G(u)$ where $u\in\Z(D)$.
Then $k(B)\le l(b)/m\le l(b)p^d$ where
\[m:=\min\bigr\{xC^{-1}x^\textup{t}:x\in\ZZ^{l(b)}\setminus\{0\}\bigl\}.\]
\end{Cor}
\begin{proof}
By the definition of $m$, the matrix $W:=\frac{1}{m}C^{-1}$ is integral positive definite. Theorem~A gives $k(B)\le\tr(WC)=l(b)/m$. For the second inequality we recall that the elementary divisors of $C$ divide $p^d$. Hence, $p^dC^{-1}$ has integral entries and $m\ge p^{-d}$. 
\end{proof}

In \cite{habil}, we referred to the \emph{Cartan method} and the \emph{inverse Cartan method} when applying \autoref{sam} and \autoref{brauer} respectively. Now we know that both methods are special cases of a single theorem. In fact, the following examples show that Theorem~A is stronger than \autoref{sam} and \autoref{brauer}:

\begin{enumerate}[(i)]
\item Let $B$ be the principal $2$-block of the affine semilinear group $G=\AGammaL(1,8)$, and let $u=1$. Then 
\[C=\begin{pmatrix}
2&.&.&1&1\\
.&2&.&1&1\\
.&.&2&1&1\\
1&1&1&4&3\\
1&1&1&3&4
\end{pmatrix}\]
and $m=\frac{1}{2}$ with the notation of \autoref{brauer}. This implies $k(B)\le 10$. On the other hand, $q(x_1,\ldots,x_5)=x_1^2+\ldots+x_5^2+x_1x_2-x_1x_5-x_2x_5-x_3x_5-x_4x_5$ in \autoref{sam} gives $k(B)\le 8$ and in fact equality holds (cf. \cite[p. 84]{KuelshammerWada}).

\item Let $B$ be the principal $2$-block of $G=A_4\times A_4$ where $A_4$ denotes the alternating group of degree $4$. Let $u=1$. Then
\[C=(1+\delta_{ij})_{i,j=1}^3\otimes (1+\delta_{ij})_{i,j=1}^3\]
and $m=9/16$ with the notation of \autoref{brauer}. Hence, we obtain $k(B)\le 16$ and equality holds. On the other hand, it has been shown in \cite[Section~3]{SambalekB2} that there is no positive definite, integral quadratic form $q$ such that $k(B)\le 16$ in \autoref{sam}.
\end{enumerate}

We give a final application where the Cartan matrix $C$ is known up to basic sets. It reveals an interesting symmetry in the formula.

\begin{Prop}
Let $B$ be a block of a finite group with abelian defect group $D$ and inertial quotient $E\le\Aut(D)$. Suppose that $u\in D$ such that $D/\langle u\rangle$ is cyclic. Then
\[k(B)\le\Bigl(\lvert\N_E(\langle u\rangle)/\C_E(u)\rvert+\frac{|\langle u\rangle|-1}{\lvert\N_E(\langle u\rangle)/\C_E(u)\rvert}\Bigr)\Bigl(\lvert\C_E(u)\rvert+\frac{|D/\langle u\rangle|-1}{\lvert\C_E(u)\rvert}\Bigr)\le|D|.\]
\end{Prop}
\begin{proof}
With the notation of Theorem~A we have $\CN=\N_E(\langle u\rangle)/\C_E(u)$. Moreover, $\overline{b}$ has defect group $D/\langle u\rangle$ and inertial quotient $\C_E(u)$. By Dade's theorem on blocks with cyclic defect groups, $l(b)=\lvert\C_E(u)\rvert$ and $C=(m+\delta_{ij})$ up to basic sets where $m:=(|D/\langle u\rangle|-1)/l(b)$ (see \cite[Theorem~8.6]{habil}). With $W=U_{l(b)}$ we obtain
\begin{align*}
k(B)&\le \Bigl(|\CN|+\frac{|\langle u\rangle-1}{|\CN|}\Bigr)\tr(WC)\\
&=\Bigl(\lvert\N_E(\langle u\rangle)/\C_E(u)\rvert+\frac{|\langle u\rangle|-1}{\lvert\N_E(\langle u\rangle)/\C_E(u)\rvert}\Bigr)\Bigl(\lvert\C_E(u)\rvert+\frac{|D/\langle u\rangle|-1}{\lvert\C_E(u)\rvert}\Bigr).
\end{align*}
The first factor is at most $|\langle u\rangle|$ and the second factor is bounded by $|D/\langle u\rangle|$. This implies the second inequality.
\end{proof}

In every example we have checked so far, Theorem~A implies Brauer's $k(B)$-Conjecture.

\section*{Acknowledgment}
I have been pursuing these formulas since my PhD in 2010 and it has always remained a challenge to prove the most general. The work on this paper was initiated in February 2018 when I received an invitation by Christine Bessenrodt to the representation theory days in Hanover. I thank her for this invitation.
The paper was written in summer 2018 while I was an interim professor at the University of Jena. I like to thank the mathematical institute for the hospitality and also my sister's family for letting me stay at their place. Moreover, I appreciate some comments on algebraic number theory by Tommy Hofmann.
The work is supported by the German Research Foundation (projects SA 2864/1-1 and SA 2864/3-1).

\end{document}